\documentclass[a4paper,11pt]{amsart}
\AtBeginDvi{}
\usepackage{amsmath,amssymb,amsthm,extarrows,mathrsfs}
\usepackage{url,here}
\usepackage[dvipdfmx]{graphicx}
\usepackage[all]{xy}
\usepackage{ifpdf}
\usepackage{bm}

\usepackage{algorithmic}
\usepackage{algorithm}

\numberwithin{equation}{section}

\DeclareMathOperator{\im}{Im}

\DeclareMathOperator{\LM}{LM}
\DeclareMathOperator{\LC}{LC}
\DeclareMathOperator{\LT}{LT}

\DeclareMathOperator{\LW}{Top}
\DeclareMathOperator{\spec}{Spec}

\newcommand{\LS}[1][\mathit{F}]{\mathrm{LSy(#1)}}
\newcommand{\LSL}[1][\mathit{F}]{\mathrm{LSyL(#1)}}
\newcommand{\LImS}[1][\mathit{F}^{(t)}_{\omega}]{\mathrm{LImSy(#1)}}
\newcommand{\Gobs}[1][\mathit{F}]{\mathrm{Gobs(#1)}}

\def\lcm{\mathrm{lcm}}
\def\Syz{\mathrm{Syz}}
\def\Spoly{\mathrm{Spoly}}

\def\A{\mathbb{A}}

\DeclareMathOperator{\ini}{in}

\newcommand{\gr}{Gr\"obner}

\theoremstyle{plain}
\newtheorem{theorem}{Theorem}[section]
\newtheorem{proposition}[theorem]{Proposition}
\newtheorem{corollary}[theorem]{Corollary}
\newtheorem{lemma}[theorem]{Lemma}

\theoremstyle{definition}
\newtheorem{definition}[theorem]{Definition}
\newtheorem{example}[theorem]{Example}
\newtheorem{remark}[theorem]{Remark}

\newtheorem*{acknowledgments}{Acknowledgments}
\newtheorem*{question}{Question}

\title[Analysis of computing GB and GD]{Analysis of computing Gr\"obner bases and Gr\"obner degenerations via theory of signatures}
\author{Yuta Kambe}
\date{}
\address{Mitsubishi Electric Corporation, Information Technology R\&D Center, 5-1-1, Ofuna, Kamakura City, 247-8501, Japan}
\email{Kambe.Yuta@bx.MitsubishiElectric.co.jp}

\begin{document}

\maketitle

\begin{abstract}
The signatures of polynomials were originally introduced by Faug\`{e}re for the efficient computation of \gr\ bases \cite{Fau02}, and redefined by Arri-Perry \cite{AP11} as the standard monomials modulo the module of syzygies. Since it is difficult to determine signatures, Vaccon-Yokoyama \cite{VY17} introduced an alternative object called guessed signatures. In this paper, we consider a module $\Gobs{}$ for a tuple of polynomials $F$ to analyse computation of \gr\ bases via theory of signatures. This is the residue module $\ini_{\prec}(\Syz(\LM(F)))/\ini_{\prec}(\Syz(F))$ defined by the initial modules of the syzygy modules with respect to the Schreyer order. We first show that $F$ is a \gr\ basis if and only if $\Gobs{}$ is the zero module. Then we show that any homogeneous \gr\ basis with respect to a graded term order satisfying a common condition must contain the remainder of a reduction of an S-polynomial. We give computational examples of transitions of minimal free resolutions of $\Gobs{}$ in a signature based algorithm. Finally, we show a connection between the module $\Gobs{}$ and \gr\ degenerations.
\end{abstract}

\tableofcontents

\section{Introduction}

The history of computing \gr\ bases began with the Buchberger's algorithm, which selects polynomials by running a multivariate division algorithm and adding them to the set of generators until it satisfies the Buchberger's criterion \cite{Buch65}. The ideas of the Buchberger's algorithm are still the basis of \gr\ basis computation algorithms, and most algorithms gradually approximate the input polynomial system to a \gr\ basis by iteratively computing the S-polynomials generated by the cancellations of the leading terms. A practical problem with this method is that the artifacts produced by the procedure are unpredictable for the choice of generators, term order, and so on. This implies a computational difficulty in applications of the \gr\ basis theory.

Our motivations in this paper are:
\begin{itemize}
	\item to obtain a quantitative cost function of a tuple of polynomials $F$ that predicts the complexity of the computation of a \gr\ basis from $F$,
	\item to answer the question of whether the S-polynomial computation is always necessary to determine a \gr\ basis, and
	\item to represent the computation of \gr\ bases in the geometrical context,
\end{itemize}
for the construction of new efficient algorithms intrinsically different from Buchberger's algorithm, such as Newton's method, midpoint method and so on, in the future. To realize it, we give an algebraic or geometric analysis of the syzygies of $F$ in the computational aspects via the theory of the signatures. Then we obtain an object $\Gobs{}$ that corresponds to the computation of a \gr\ basis from $F$ and a \gr\ degeneration of $F$. And we prove that remainders of divisions of S-polynomials must be determined to obtain homogeneous \gr\ bases with respect to graded orders.

Let $R = K [x_1,\ldots, x_n]$ be the polynomial ring  with a term order $<$ over a field $K$, $F = (f_1,f_2,\ldots,f_m)$ a tuple of elements in $R$, and $I$ the ideal generated by $F$. By $R^m = \oplus_{i=1}^m Re_i$ we denote the free $R$-module with the basis $(e_1,e_2,\ldots,e_m)$ corresponding to $F$. Assume that $R^m$ equips a term order $\prec$. The signature $S(f)$ of a non-zero element $f$ in $I$ is defined as
\[ S(f) = \min_{\prec} \{ \LM(u) \mid u \in R^m,\ \bar{u} = f\}, \]
where $\bar{u}$ is the image of $u$ under the canonical surjection $R^m \rightarrow I \rightarrow 0$ (see also Definition \ref{def.signature}, Proposition \ref{prop.LMSyz_defines_signatures}). Faug\`{e}re first introduced the concept of signatures in his $F_5$ algorithm for efficient computation of \gr\ bases by avoiding reductions to zero \cite{Fau02}. Several researchers proposed many variants of the $F_5$ algorithm, nowadays called signature based algorithms. Arri-Perry introduced another definition of the signatures to give a proof of the termination and correctness of the $F_5$ algorithm or signature based algorithms for any input \cite{AP11}. It is difficult to determine the signature for a general polynomial without a \gr\ basis of $I$ or the syzygy module $\Syz(F)$. Vaccon-Yokoyama defined the ``guessed'' signatures of the S-polynomials as an alternative object of signatures \cite{VY17}. The guessed signatures are only determined from the computational history of the running instance. Then they made a simple implementation of a signature based algorithm. In this paper we introduce a definition of guessed signatures that is different from \cite{VY17}. We define the guessed signatures for pairs $(x^{\alpha}e_i, x^{\beta}e_j)$ of monomials in $R^m$ such that $x^{\alpha}\LM(f_i) = x^{\beta} \LM(f_j)\ (i<j)$ as the monomials $x^{\beta}e_j$ in the second components (Definition \ref{def.guessed_signatures}).

If we attach the Schreyer order on $R^m$ (Definition \ref{def.Schreyer_order}), the guessed signature of a pair $(x^{\alpha}e_i, x^{\beta}e_j)$ is the leading monomial of $x^{\alpha}e_i - x^{\beta} e_j$. In fact, the guessed signature of a pair $(x^{\alpha}e_i, x^{\beta}e_j)$ is not always the signature of the S-polynomial $\frac{1}{\LC(f_i)}x^{\alpha}f_i - \frac{1}{\LC(f_j)}x^{\beta}f_j$. It partly depends on whether the reduction of the S-polynomial is zero or not. From this point of view, in this paper we suppose that the difference between the set of guessed signatures and the set of signatures might predict the behavior to computations of \gr\ bases from $F$, and then we focus on this difference. From the Schreyer's theorem, the set of guessed signatures is the set of the leading monomials $\LM(\Syz(\LM(F)))$ of the syzygy module of the tuple $\LM(F) = (\LM(f_1),\ldots, \LM(f_m))$ \cite[Theorem 15.10]{Eis95}. Then our main target is the residue module
\[ \Gobs{} = \ini_{\prec}(\Syz(\LM(F)))/\ini_{\prec}(\Syz(F)). \]

From now on we always attach the Schreyer order on $R^m$. Our contributions in this paper are the following.
\begin{itemize}
	\item[(A)] We give a criterion for \gr\ bases: $F$ is a \gr\ basis if and only if $\Gobs{} = 0$ (Theorem \ref{thm.criterion_by_signature}).
	\item[(B)] We show that for any homogeneous \gr\ basis $G$ of $I$ including $F$ with respect to a graded term order $<$, $G$ contains an element $g$ such that $\LM(g) = \LM(r)$, where $r$ is the remainder of a reduction of an S-polynomial. If $G$ satisfies some common condition, then $g = cr\ (\exists\, c \in K)$ (Corollary \ref{cor.S-poly_is_needed}).
	\item[(C)] We give examples of transitions of $\Gobs{}$ in a signature based algorithm (Section \ref{sec.example}).
	\item[(D)] We find a closed subscheme $X$ in $\spec R \times_K \mathbb{A}_K^1$ and direct summand $N(F)$ of $\Gobs{}$ such that $X$ is a flat deformation of $\spec R/I$ to $\spec R/\langle \LM(F) \rangle$ over $\mathbb{A}_K^1$ if and only if $N(F) = 0$ (Theorem \ref{thm.criterion_by_the_image_of_SyzFt}, Lemma \ref{lem.LSF_subset_LImS}).
\end{itemize}

For (A), a key lemma is the following (see also Lemma \ref{lem.Signature_differ_LM}).
\begin{lemma}\label{lem.intro_signature_differ_LM}
For any element $f$ in $I$, the condition
\[ \LM(f) \not\in \langle \LM(f_1),\ldots \LM(f_m) \rangle\]
implies that
\[ S(f) \in \LM(\Syz(\LM(F))) \setminus \LM(\Syz(F)). \] 
\end{lemma}

(B) is based on Lemma \ref{lem.intro_signature_differ_LM}. Let us consider about finding an element of the leading monomial not in $\langle \LM(f_1),\ldots, \LM(f_m) \rangle$. Let $f_{m+1}$ be an element of $I$ such that $\LM(f_{m+1}) \not\in \langle \LM(f_1),\ldots, \LM(f_m) \rangle$ and put $F' = (f_1,f_2,\ldots, f_m,f_{m+1})$. Assume that $f_{m+1} = \bar{u}$ for an element $u$ in $R^m$ and $\LM(u) = S(f)$. By Lemma \ref{lem.intro_signature_differ_LM}, the equivalent class of $S(f_{m+1})$ in $\Gobs{}$ is not zero. On the other hand, since $u - e_{m+1} \in \Syz(F')$ and $\LM(u-e_{m+1}) = \LM(u)$ (see Lemma \ref{lem.Signature_differ_LM}), we can show that the equivalent class of $S(f_{m+1})$ in $\Gobs[\mathit{F}']$ is zero. Then one may interpret that finding an element $f_{m+1}$ that the leading monomial not in $\langle \LM(f_1),\ldots, \LM(f_m) \rangle$ is vanishing a non-zero element of $\Gobs{}$. If $F$ consists of homogeneous elements and the term order $<$ on $R$ is graded lexicographic order or graded reverse lexicographic order, one may consider that the signature $S(f)$ is an index of the computational cost of representing $f$ by $F$, since degrees are a factor of the complexity of computing polynomials \cite{MM84,Dub90,Giu05,BFSY05}. Therefore a naive idea to compute \gr\ bases efficiently is to choose polynomials of small signatures. In fact, several signature based algorithms follow this idea \cite{AP11,VY17,Sak20} (see also Algorithm \ref{alg:minimum_SBA}). Then we identify the polynomials of the signature that is smallest in $\Gobs{}$.

\begin{theorem}[Theorem \ref{thm.why_need_s-poly_division}]\label{thm.intro_remainder}
Assume that $F$ is not a \gr\ basis. For any element $f$ in $I$, if it holds that $\LM(f) \not \in \langle \LM(f_1),\ldots, \LM(f_m) \rangle$ and the signature $S(f)$ is minimum in $\LM(\Syz(\LM(F))) \setminus \LM(\Syz(F))$, then it satisfies that $\LM(f) = \LM(r)$, where $r$ is the remainder of any division of an S-polynomial of the signature $S(f)$. If the all terms of $f$ and $r$ are not in $\langle \LM(f_1), \ldots, \LM(f_m) \rangle$, then $f = cr$ for some $c \in K$.
\end{theorem}

Let us assume again that $F$ consists of homogeneous elements and the term order $<$ on $R$ is graded lexicographic order or graded reverse lexicographic order. What would happen if we choose a homogeneous polynomial $f_{m+1}$ that satisfies
\[ S(f_{m+1}) \neq \min \left[ \LM(\Syz(\LM(F))) \setminus \LM(\Syz(F)) \right]? \]
In fact, it will happen that
\[ \begin{aligned} s &= \min \left[ \LM(\Syz(\LM(F))) \setminus \LM(\Syz(F))\right]\\ 
& = \min \left[ \LM(\Syz(\LM(\mathit{F} \cup \{f_{m+1}\}))) \setminus \LM(\Syz(\mathit{F} \cup \{f_{m+1}\}))\right] \end{aligned}\]
(Theorem \ref{thm.minimum_still_minimum}). Namely, $s$ do not vanish in $\Gobs[\mathit{F}\cup\{\mathit{f}_{m+1}\}]$ and then $F \cup \{ f_{m+1}\}$ can not be a \gr\ basis. Therefore we obtain the following theorem that gives the necessity of the S-polynomial computation.

\begin{theorem}[Corollary \ref{cor.S-poly_is_needed}]\label{thm.intro_S_need}
For any homogeneous \gr\ basis $G$ of $I$ including $F$ with respect to a graded term order, there exist a subset $F'$ and an element $g \in G$ such that $\LM(g) = \LM(r)$, where $r$ is the remainder of any division of an S-polynomial of the signature $s$ with respect to $F'$. If the non-leading terms of elements of $G$ are not in $\langle \LM(G) \rangle$, then $g = cr$ for some $c \in K$.
\end{theorem}

About (C), as mentioned above, some signature based algorithms can be intuitively thought of as methods that attempt to reduce the size of $\Gobs{}$ by annihilating the smallest elements. However, in Section \ref{sec.example}, we observe examples of transitions of $\Gobs{}$ in an implementation of a signature based algorithm, and we find examples that the sequence of $\Gobs{}$ does not monotonically go to the zero-module in the procedure. On the other hand, observing such examples leads to the conjecture that, in some cases, the first Betti number of $\Gobs{}$ represents the phase of the monomial ideal generated by $\LM(F)$. More precisely, some examples satisfy the statement that if the first Betti number increases in a step, then the new leading monomial found in that step divides another leading monomial of the generators (Example \ref{ex.not_decrease}, Example \ref{ex.example2}, Example \ref{ex.finite_field}). However, the above statement is not true in Example \ref{ex.example3}. Furthermore, in Example \ref{ex.example3}, $\Gobs{}$ is generated by a single equivalent class for the input $F$, nevertheless the instance does not terminate by a single step. We still do not know what is going on in the background of all this.

About (D), we show that $\Gobs{}$ contains flatness obstructions of a family introduced from $F$ in the context of \gr\ degenerations. Then we call $\Gobs{}$ the module of \gr ness obstructions of $F$ in this paper. Let us recall \gr\ degenerations. We call a closed subscheme $X$ in $\spec R \times_{K} \spec K[t]$ a \gr\ degeneration of $\spec R/I$ if the projection $X \rightarrow \spec K[t]$ is flat, generic fibers $X_t$ of the projection over $t \neq 0$ are isomorphic to $\spec R/I$ and the special fiber $X_0$ at $t = 0$ is isomorphic to $\spec R/\ini_{<} (I)$. There exists a \gr\ degeneration constructed from a weighting on variables \cite{Bay82,Eis95}. \gr\ degenerations are used in studies of degenerations of varieties, homological invariants, Hilbert schemes and so on \cite{Har66,KM05toric,LR11rat,CV20square,Kam22}. Our main theorem about the relationship between $\Gobs{}$ and \gr\ degenerations is the following.

\begin{theorem}[Theorem \ref{thm.criterion_by_the_image_of_SyzFt}, Lemma \ref{lem.LSF_subset_LImS}]\label{thm:intro_main2}
There exists a closed subscheme $X$ in $\spec R \times_{K} \spec K[t]$ and a direct summand $N(F)$ of $\Gobs{}$ such that
\begin{itemize}
	\item generic fibers of the projection $X \rightarrow \spec K[t]$ over $t \neq 0$ are isomorphic to $\spec R/I$, the special fiber at $t = 0$ is isomorphic to $\spec R/\langle \LM(F) \rangle$,
	\item the projection $X \rightarrow \spec K[t]$ is flat if and only if $N(F) =0$.
\end{itemize}
\end{theorem}

\section{Preliminary}

Let $K$ be a field. Let $R= K[x_1,\ldots, x_n]$ be the polynomial ring over $K$ in $n$ variables attached a term order $<$. Here a term order means a total order $<$ of monomials in $R$ such that $1 < m$ for any monomial $m \neq 1$ and $m < n$ implies $ml < nl$ for any monomials $m,n,l$. We say a term order $<$ is \textit{graded} if $m < n$ for any monomials $m,n$ such that $\deg m < \deg n$ for the ordinal total degree of $R$. We use the following notation:
\begin{itemize}
	\item $\langle A \rangle$: the ideal generated by $A$ in $R$,
	\item $\LM(f)$: the leading monomial of $f$,
	\item $\LC(f)$: the leading coefficient of $f$,
	\item $\LT(f) = \LC(f) \LM(f)$: the leading term of $f$,
	\item $x^{\alpha} = x_1^{\alpha_1}x_2^{\alpha_2}\cdots x_n^{\alpha_n}$ for a vector $\alpha = (\alpha_1,\alpha_2,\ldots,\alpha_n)$.
\end{itemize}
We always consider a fixed tuple of polynomials $F = (f_1,\ldots,f_m)$ such that $f_i \neq 0\ (i = 1,\ldots, m)$ unless otherwise noted.

In this paper, a division means a reduction by $F$ such that the remainder is $0$ or has no terms in $\langle \LM(f_1) ,\ldots, \LM(f_m) \rangle$.

\begin{definition}\label{def.reduction}
For any polynomial $f$ in $R$, there exist polynomials $h_1$,$\ldots$, $h_m$ and $r$ in $R$ such that
\[ f = \sum_{i=1}^m h_i f_i + r,\ \LM(h_if_i) \leq \LM(f), \]
and $r = 0$ or the all terms of $r$ are not in $\langle \LM(f_1) ,\ldots, \LM(f_m) \rangle$. We call this form a \textit{division} of $f$ with $F$. We also call $h_1,\ldots, h_m$ the \textit{quotient} and $r$ the \textit{remainder} of this division of $f$ with $F$.
\end{definition}

Let $I = \langle F \rangle$ be the ideal generated by $F$ in $R$. We call the ideal $\langle \LM(f) \mid f \in I\setminus \{0\} \rangle$ the initial ideal of $I$ and denote it by $\ini_{<}(I)$. We say $F$ is a \gr\ basis if the initial ideal $\ini_{<}(I)$ is generated by the tuple $\LM(F) = (\LM(f_1), \ldots, \LM(f_m))$. For the elementary of \gr\ bases, see \cite[Section 15]{Eis95}.

Let $R^m = \oplus_{i=1}^m Re_i$ be the free $R$-module of rank $m$ with the basis $(e_1,\ldots,e_m)$. A \textit{monomial in $R^m$} is an element of the form $x^{\alpha}e_i$. In this paper, we always attach the following order on $R^m$.

\begin{definition}\label{def.Schreyer_order}
The \textit{Schreyer order} on $R^m$ is the order of monomials in $R^m$ such that
\[ x^{\alpha} e_i \prec x^{\beta} e_j \Leftrightarrow \begin{aligned}&x^{\alpha} \LM(f_i) < x^{\beta} \LM(f_j)\\
&\text{or}\ (x^{\alpha} \LM(f_i) = x^{\beta} \LM(f_j)\ \text{and}\ i < j).
\end{aligned} \]
\end{definition}

Let $u$ be a non-zero element in $R^m$. The \textit{leading monomial} of $u$ is the largest monomial with non-zero coefficient occurring in $u$. We define the \textit{leading coefficient} and \textit{leading term} as the same. We use the following notation:

\begin{itemize}
	\item $\LM(u)$: the leading monomial of $u$,
	\item $\LC(u)$: the leading coefficient of $u$,
	\item $\LT(u) = \LC(u) \LM(u)$: the leading term of $u$,
	\item $\LM(M) = \{\LM(u) \mid u \in M \}$ for a subset $M$ in $R^m$,
	\item $\langle N \rangle$: the $R$-submodule generated by a subset $N$ in $R^m$.
\end{itemize}

Let $M$ be an $R$-submodule in $R^m$. The \textit{initial module} $\ini_{\prec}(M)$ of $M$ is the $R$-submodule in $R^m$ generated by $\LM(M)$. A set of generators $V$ of $M$ is a \textit{\gr\ basis} of $M$ if the initial module $\ini_{\prec}(M)$ is generated by $\LM(V)$.

Let us define the syzygies.

\begin{definition}\label{def.Value_Notation}
The notation $\bar{u}$ for $u$ denotes the value of the $R$-module morphism
\[
\begin{array}{ccc}
 R^m & \rightarrow & I\\
 e_i & \mapsto & f_i
\end{array}
\]
at $u$. If $\bar{u} = 0$, then we say $u$ is a \textit{syzygy} of $F$. The \textit{syzygy module} of $F$ is the kernel of the above morphism. We denote the syzygy module of $F$ by $\Syz(F)$.
\end{definition}

In general, generators of the syzygy module $\Syz(F)$ depend on $F$ and need precise computation to determine. On the other hand, generators of the syzygy module $\Syz(\LM(F))$ is theoretically determined with an explicit form by the Schreyer's theorem.

\begin{theorem}[{\cite[Theorem 15.10]{Eis95}}]\label{thm.GB_of_initial_syzygy}
Let
\[m^{(i,j)}_i = \frac{\lcm(\LM(f_i),\LM(f_j))}{\LM(f_i)},m^{(i,j)}_j = \frac{\lcm(\LM(f_i),\LM(f_j))}{\LM(f_j)}\]
for distinct indexes $i,j$. Then the set
\[ \left\{ \left. m^{(i,j)}_i e_i - m^{(i,j)}_je_j \right|  i<j \right\}\]
is a \gr\ basis of $\Syz(\LM(F))$. In particular, the initial module of $\Syz(\LM(F))$ is generated by the set $\{ m^{(i,j)}_j e_j \mid i < j \}$.
\end{theorem}

\begin{proposition}\label{prop.LSFvsLSLF}
It holds that $\LM(\Syz(F)) \subset \LM(\Syz(\LM(F)))$.
\end{proposition}

\begin{proof}
For any $u \in \Syz(F)$, denote
\[ u = \sum_{\alpha,i} c_{\alpha,i} x^{\alpha}e_i, \]
where $c_{\alpha,i} \in K$. Consider $x^{\xi} = \max \{ x^{\alpha}\LM(f_i) \mid c_{\alpha,i} \neq 0 \}$. Let us divide $u$ into the following two parts:
\[ u_0 = \sum_{x^{\alpha}\LM(f_i) = x^{\xi}} c_{\alpha,i} x^{\alpha}e_i,\ u_1 = \sum_{x^{\alpha}\LM(f_i) < x^{\xi}} c_{\alpha,i} x^{\alpha}e_i.  \]
By definition of the Schreyer order, we have $\LM(u) = \LM(u_0)$, thus it is enough to show that $\LM(u_0) \in \LM(\Syz(\LM(F)))$. Let us compute $\overline{u_0}$ as the following:
\[ \begin{aligned} \overline{u_0} &= \sum_{x^{\alpha}\LM(f_i) = x^{\xi}} c_{\alpha,i} x^{\alpha}f_i\\
&= \sum_{x^{\alpha}\LM(f_i) = x^{\xi}} \left(c_{\alpha,i}\LC(f_i)\right) x^{\alpha} \LM(f_i) + \sum_{x^{\alpha}\LM(f_i) = x^{\xi}} c_{\alpha,i} x^{\alpha}(f_i - \LT(f_i)).\end{aligned} \]
Since the second sum in the above consists of terms smaller than $x^{\xi}$, the term of $\bar{u} = \overline{u_0} + \overline{u_1}$ at $x^{\xi}$ is $\sum_{x^{\alpha}\LM(f_i) = x^{\xi}} \left(c_{\alpha,i}\LC(f_i)\right) x^{\alpha}\LM(f_i)$ which must be $0$. Then the element
\[ v = \sum_{x^{\alpha}\LM(f_i) = x^{\xi}} \left(c_{\alpha,i}\LC(f_i)\right) x^{\alpha}e_i\]
is a syzygy of $\LM(F)$. Therefore we have
\[ \begin{aligned} \LM(u_0) &= \max_{\prec} \left\{ x^{\alpha}e_i \left| \begin{aligned}c_{\alpha,i} \neq 0,\\
\LM(x^{\alpha}f_i) = x^{\xi} \end{aligned} \right.\right\}\\
&= \max_{\prec} \left\{ x^{\alpha}e_i \left| \begin{aligned}c_{\alpha,i}\LC(f_i) \neq 0,\\
\LM(x^{\alpha}f_i) = x^{\xi} \end{aligned} \right. \right\}  = \LM(v) \in \LM(\Syz(\LM(F))). \end{aligned}\]
\end{proof}

\section{Signatures and guessed signatures}\label{sec.Sig_and_Guess_Sig}

We recall the definition of the signatures given in \cite{Fau02,AP11}. 

\begin{definition}\label{def.signature}
Let $f$ be a non-zero element in $I$. The \textit{signature} of $f$ is the minimum element of $\{ \LM(u) \mid u \in R^m, \bar{u} = f \}$. We denote the signature of $f$ by $S(f)$.
\end{definition}

\begin{proposition}[\cite{AP11}]\label{prop.LMSyz_defines_signatures}
The set of signatures $\{ S(f) \mid f \in I \setminus \{0\} \}$ equals to the following set of monomials: $\{ s \mid \text{$s$ is a monomial in $R^m$},\ s \not\in \LM(\Syz(F))\}$. In particular, the set of the equivalent classes of the signatures is a basis of the residue module $R^m/\Syz(F)$ as a $K$-linear space.
\end{proposition}

As easiest example of signatures, one may hope that $S(f_i) = e_i$. However, it is wrong in general. For example, assume $F=(f_1,f_2,f_3)$, $f_3 = f_1+f_2$ and $\LM(f_1) < \LM(f_2)$, then the signature of $f_3$ is not $e_3$. Indeed, put $u = e_1+e_2$. We have $\bar{u} = f_3$ and $\LM(u) = e_2$. Thus the signature of $f_3$ is less than or equal to $e_2$. Since we attach the Schreyer order on $R^m$, we have $e_2 < e_3$. Therefore we obtain $S(f_3) < e_3$. Note that, in general, we need a \gr\ basis of $\Syz(F)$ to determine the signature $S(f)$ of given polynomial $f$.

As a more reasonable object than the signatures, we introduce the guessed signatures.

\begin{definition}\label{def.guessed_signatures}
An \textit{S-pair} is a pair of monomials $(x^{\gamma}e_k,x^{\delta}e_{\ell})$ such that $k < \ell$ and $x^{\gamma}\LM(f_k) = x^{\delta} \LM(f_{\ell})$. We denote S-pairs as $p = (x^{\gamma}e_k,x^{\delta}e_{\ell})$. The \textit{S-polynomial} of $p = (x^{\gamma}e_k,x^{\delta}e_{\ell})$ denoted by $\Spoly(p)$ is the polynomial
\[ \Spoly(p) = \frac{1}{\LC(f_k)}x^{\gamma}f_k - \frac{1}{\LC(f_{\ell})}x^{\delta} f_{\ell}.\]
For an S-pair $p=(x^{\gamma}e_k,x^{\delta}e_{\ell})$, we call the second component $x^{\delta}e_{\ell}$ the \textit{guessed signature} of $p$. We denote the guessed signature of $p$ by $\hat{S}(p)$. We say an S-pair $p = (x^{\gamma}e_k,x^{\delta}e_{\ell})$ is \textit{standard} if it satisfies that $x^{\gamma}\LM(f_k) = x^{\delta}\LM(f_{\ell}) = \lcm(\LM(f_k),\LM(f_{\ell}))$.
\end{definition}

\begin{remark}\label{rem.deff_from_origin}
The original definition of guessed signature is not as in Definition \ref{def.guessed_signatures}. We note the original definition that previous studies (for example, \cite{AP11,VY17,Sak20}) used in the following: fix a tuple $F$ as a set of generators of the ideal $I$ and consider a set $G = \{g_1, g_2, \ldots, g_b \}$ of elements in $I$ including $F$, we call a pair of generators $(g_i,g_j)$ a S-pair of $G$ if $i \neq j$. An S-pair $(g_i,g_j)$ is \textit{pseudo regular} if
\[ m_i^{(i,j)}S(g_i) \neq m_j^{(i,j)}S(g_j). \]
The \textit{guessed signature} of a pseudo regular S-pair $(g_i,g_j)$ is the maximum element of the set $\{m_i^{(i,j)}S(g_i), m_j^{(i,j)}S(g_j)\}$.

In our definition (Definition \ref{def.guessed_signatures}), we only consider the situation of $G = F$, omit hypothesis on pseudo regularity, and use $x^{\delta}e_{\ell}$ as the guessed signature instead of $x^{\delta}S(f_{\ell})$ for convenience in the latter.
\end{remark}

Since it holds that
\[ \Spoly(x^{\gamma}e_k,x^{\delta}e_{\ell}) = \overline{\left(\frac{1}{\LC(f_k)}x^{\gamma}e_k - \frac{1}{\LC(f_{\ell})}x^{\delta}e_{\ell}\right)},\]
one may guess that the signature of the S-polynomial is $x^{\delta}e_{\ell}$. This is the reason why we call $x^{\delta}e_{\ell}$ the ``guessed'' signature. In fact, the equality $S(\Spoly(p)) = \hat{S}(p)$ is a non-trivial condition to determine if $F$ is a \gr\ basis or not.

\begin{theorem}\label{thm.criterion_by_signature}
The following are equivalent.

\begin{itemize}
	\item[(a)] Tuple $F =(f_1,\ldots,f_m)$ is a \gr\ basis.
	\item[(b)] For any $S$-pair $p$, the guessed signature $\hat{S}(p)$ is not the signature $S(\Spoly(p))$.
	\item[(c)] For any standard $S$-pair $p$, the guessed signature $\hat{S}(p)$ is not the signature $S(\Spoly(p))$.
	\item[(d)] The equality $\LM(\Syz(F)) = \LM(\Syz(\LM(F)))$ holds.
	\item[(e)] For any non-zero element $f \in I$, the leading monomial $\LM\left(\overline{S(f)}\right)$ equals to the leading monomial $\LM(f)$.
\end{itemize}
\end{theorem}

Here we note the mean of the condition (e). Let $u = \sum_{\alpha, i} c_{\alpha,i} x^{\alpha} e_i$ be an element of $R^m$ such that $\overline{u} = f$ and $\LM(u) = S(f)$. Assume that $S(f) = x^{\beta}e_j$ and put $x^{\xi} = \LM\left(\overline{S(f)}\right) = x^{\beta} \LM(f_j)$. Then by definition of the Schreyer order we have $x^{\xi} = \max \{ x^{\alpha}\LM(f_i) \mid c_{\alpha,i} \neq 0 \}$. We divide $f$ into the following two parts:
\[ f = \overline{u_0} + \overline{u-u_0} = \sum_{x^{\xi}= x^{\alpha}\LM(f_i)} c_{\alpha,i} x^{\alpha} f_i + \sum_{x^{\xi} > x^{\beta}\LM(f_i)} c_{\beta,j} x^{\beta} f_i.\]
Therefore the inequality $\LM\left(\overline{S(f)} \right) \succeq \LM(f)$ always holds, and we have $\LM(f) \in \langle \LM(f_1), \ldots, \LM(f_m) \rangle$ if the equality $\LM\left(\overline{S(f)} \right) = \LM(f)$ holds.

We proof Theorem \ref{thm.criterion_by_signature} after introducing some lemmas we need.

\begin{lemma}\label{lem.guessed_is_generators}
The set of the guessed signatures $\{\hat{S}(p) \mid \text{$p$ is a S-pair}\}$ equals to $\LM(\Syz(\LM(F)))$. Moreover, the initial module of $\Syz(\LM(F))$ is generated by a subset $\{ \hat{S}(p) \mid \text{$p$ is a standard S-pair}\}$.
\end{lemma}

\begin{proof}
The latter part is clear from Theorem \ref{thm.GB_of_initial_syzygy}. Let $L$ be the set of the guessed signature of standard S-pairs. For any S-pair $p = (x^{\gamma}e_k,x^{\delta}e_{\ell})$, there exists a monomial $x^{\lambda}$ such that
\[x^{\gamma}\LM(f_k) = x^{\delta}\LM(f_{\ell}) = x^{\lambda} \lcm(\LM(f_k),\LM(f_{\ell})).\]
Assume that $\lcm(\LM(f_k),\LM(f_{\ell})) = x^{\alpha}\LM(f_k) = x^{\beta}\LM(f_{\ell})$. We have $x^{\delta} = x^{\lambda}x^{\beta}$ and then $\hat{S}(p) = x^{\delta}e_{\ell} = x^{\lambda}\hat{S}(x^{\alpha}e_k,x^{\beta}e_{\ell})$. Therefore the guessed signature $\hat{S}(p)$ is a multiple of an element of $L$ and then an element of $\LM(\Syz(\LM(F))$. Conversely, for any element of $u \in \Syz(\LM(F))$, there exist a monomial $x^{\lambda}$ and an element $\hat{S}(x^{\gamma}e_k,x^{\delta}e_{\ell})$ in $L$ such that $\LM(u) = x^{\lambda} \hat{S}(x^{\gamma}e_k,x^{\delta}e_{\ell}) = \hat{S}(x^{\lambda}x^{\gamma}e_k,x^{\lambda}x^{\delta}e_{\ell})$. Therefore $\LM(u)$ is the guessed signature of a S-pair $(x^{\lambda}x^{\gamma}e_k,x^{\lambda}x^{\delta}e_{\ell})$.
\end{proof}

\begin{lemma}\label{lem.Signature_differ_LM}
Let $f$ be an element of $I \setminus \{0\}$. If it holds that $\LM\left(\overline{S(f)}\right) > \LM(f)$, then the signature $S(f)$ of $f$ is an element of $\LM(\Syz(\LM(F)))$.
\end{lemma}

\begin{proof}
Let $u = \sum_{\alpha, i} c_{\alpha,i} x^{\alpha} e_i$ be an element of $R^m$ such that $\bar{u} = f$ and $\LM(u) = S(f)$. Assume that $S(f) = x^{\beta}e_j$ and put $x^{\xi} = \LM\left(\overline{S(f)}\right) = x^{\beta} \LM(f_j)$. Then, by definition of the Schreyer order, we have
\[ x^{\xi} = \max \{ x^{\alpha}\LM(f_i) \mid c_{\alpha,i} \neq 0 \}.\] Therefore as the proof of Proposition \ref{prop.LSFvsLSLF}, putting
\[ u_0 = \sum_{x^{\xi}= x^{\alpha}\LM(f_i)} c_{\alpha,i} x^{\alpha} e_i,\]
we obtain $\LM(u_0) = S(f)$ and
\[f = \sum_{x^{\xi}= x^{\alpha}\LM(f_i)} c_{\alpha,i} x^{\alpha} f_i + \sum_{x^{\xi} > x^{\beta}\LM(f_j)} c_{\beta,j} x^{\beta} f_j.\]
Hence it is enough to show that $\LM(u_0) \in \LM(\Syz(\LM(F)))$. Since $x^{\xi} > \LM(f)$, it holds that $\sum_{x^{\xi} = x^{\alpha}\LM(f_i)} \left(c_{\alpha,i}\LC(f_i)\right) x^{\alpha} \LM(f_i) = 0$. Then we have
\[ \sum_{x^{\xi} = x^{\alpha}\LM(f_i)} \left(c_{\alpha,i}\LC(f_i)\right) x^{\alpha}e_i \in \Syz(\LM(F)).\]
Using the same logic in the proof of Proposition \ref{prop.LSFvsLSLF}, we obtain $\LM(u_0) \in \LM(\Syz(\LM(F)))$. 
\end{proof}

\begin{lemma}\label{lem.signature_or_not}
Let $f$ be an element of $I\setminus \{0\}$ and $u$ an element of $R^m$ such that $\bar{u} = f$. The equality $\LM(u) = S(f)$ holds if and only if $\LM(u)$ is not an element of $\LM(\Syz(F))$.
\end{lemma}

\begin{proof}
By definition of signatures, inequality $\LM(u) \succeq S(f)$ always holds. If the equality $\LM(u) = S(f)$ holds, then we have $\LM(u) \not\in \LM(\Syz(F))$ from Proposition \ref{prop.LMSyz_defines_signatures}. Conversely, if it holds that $\LM(u) > S(f)$, let $v$ be an element of $R^m$ such that $\bar{v} = f$ and $\LM(v) = S(f)$. Then $u-v$ is a syzygy of $F$. Therefore we obtain that $\LM(u) = \LM(u-v) \in \LM(\Syz(F))$.
\end{proof}

\begin{proof}[Proof of Theorem \ref{thm.criterion_by_signature}]
[(a) $\implies$ (b)] If $F$ is a \gr\ basis, then for any S-pair $p = (x^{\gamma}e_k,x^{\delta}e_{\ell})$, there exist polynomials $h_1,h_2,\ldots,h_t$ in $R$ such that
\[ \Spoly(p) = \frac{1}{\LC(f_k)}x^{\gamma}f_k - \frac{1}{\LC(f_{\ell})}x^{\delta}f_{\ell} = \sum_{i=1}^t h_i f_i,\ \LM(h_if_i) < \LM(x^{\delta}f_{\ell}) \]
by taking the normal form of $\Spoly(p)$ with $F$. Let
\[ u = \frac{1}{\LC(f_k)}x^{\gamma}e_k - \frac{1}{\LC(f_{\ell})}x^{\delta}e_{\ell} -\sum_{i=1}^t h_i e_i. \]
Since $\LM(h_if_i) < \LM(x^{\delta}f_{\ell})$, we have $\LM(u) = x^{\delta}e_{\ell} \in \LM(\Syz(F))$. Therefore the guessed signature of $p$ is not the signature of $\Spoly(p)$ since any signature is not an element of $\LM(\Syz(F))$ (Proposition \ref{prop.LMSyz_defines_signatures}).

[(b) $\implies$ (c)] It is trivial.

[(c) $\implies$ (d)] From Lemma \ref{lem.guessed_is_generators}, the initial module $\ini_{\prec}\left(\Syz(\LM(F))\right)$ is generated by a set $\{ \hat{S}(p) \mid \text{$p$ is a standard S-pair}\}$. This set is a subset of $\LM(\Syz(F))$ from the assumption and Lemma \ref{lem.signature_or_not}. Therefore the equality $\LM(\Syz(F)) = \LM(\Syz(\LM(F)))$ holds.

[(d) $\implies$ (e)] For any non-zero element $f \in I$, the signature $S(f)$ is not an element of $\LM(\Syz(F))$ (Proposition \ref{prop.LMSyz_defines_signatures}). From (d), the signature $S(f)$ is also not an element of $\LM(\Syz(\LM(F)))$, therefore it holds that $\LM\left(\overline{S(f)}\right) = \LM(f)$ by Lemma \ref{lem.Signature_differ_LM}.

[(e) $\implies$ (a)] For any non-zero element $f \in I$, we have $\LM(f) = \LM\left(\overline{S(f)}\right) \in \langle \LM(f_1),\ldots, \LM(f_m) \rangle$. Therefore the tuple $F$ is a \gr\ basis.
\end{proof}

As a consequence of Theorem \ref{thm.criterion_by_signature}, we find an algebraic obstacle where the tuple of generators $F$ is a \gr\ basis. Namely, for a tuple of generators $F$,
\[ \text{$F$ is a \gr\ basis} \Leftrightarrow \langle \LM(\Syz(\LM(F)) \rangle / \langle \LM(\Syz(F)) \rangle = 0. \]
In latter, we put
\[ \LS{} = \LM(\Syz(F)),\ \LSL{} = \LM(\Syz(\LM(F))) \]
for short. Moreover, we put
\[ \Gobs{} = \langle \LSL{} \rangle/ \langle \LS{} \rangle = \ini_{\prec}(\Syz(\LM(F)))/\ini_{\prec}(\Syz(F))\]
and call it the \textit{module of \gr ness obstructions} of $F$.

We can compute the smallest non-zero element of $\LSL{}\setminus\LS{}$ using a step-by-step method.

\begin{proposition}\label{prop.find_minimum_guessed_signature}
Let $s_i$ be the $i$-th smallest element of the set
\[ \{\hat{S}(p) \mid \text{$p$ is a standard S-pair} \}.\]
Let $p$ be a standard S-pair such that $\hat{S}(p) = s_i$. Assume that $i = 1$ or $s_1,s_2,\ldots, s_{i-1} \in \LS{}\ (i \geq 2)$. Then $s_i \in \LS{}$ if and only if the reminder of any division of the S-polynomial $\Spoly(p)$ with $F$ is $0$.
\end{proposition}

\begin{proof}
Assume that $p = (x^{\gamma}e_k,x^{\delta}e_{\ell})$. Let $h_1,\ldots,h_m$ be the quotients and $r_i$ the remainder of any division of $\Spoly(p)$ with $F$. Then it holds that
\[ \Spoly(p) = \sum_{t=1}^a h_t f_t + r,\ \LM(h_t f_t) <  x^{\delta}\LM(f_{\ell}) \]
and $r=0$ or $\LM(r)$ does not belong to $\langle \LM(F) \rangle$. Put $u = \frac{x^{\gamma}}{\LC(f_{k})} e_{k} - \frac{x^{\delta}}{\LC(f_{\ell})}e_{\ell} - \sum_{t=1}^a h_{t} e_{t}$. We have $\LM(u) = s_i$ and $\overline{u} = r$. It implies that $S(r) \leq s_i$ if $r \neq 0$.

If $r = 0$, then the element $u$ is a syzygy of $F$. Therefore we have $s_i \in \LS{}$.

Let us show the converse. If $i = 1$ and $r \neq 0$, then the signature $S(r)$ is an element of $\LSL{}$ from Lemma \ref{lem.Signature_differ_LM} since $\LM(r) \not \in \langle \LM(F) \rangle$. Therefore we obtain $s_1 = S(r)$ and $s_1 \not \in \LS{}$ since $s_1$ is the minimum element of $\LSL{}$. If $i \geq 2$ and $r \neq 0$, then the signature $S(r)$ is also an element of $\LSL{}$. Since $S(r) \leq s_i$, there exists an index $j$ smaller than or equal to $i$ such that $s_j | S(r)$ (note that $\langle \LSL{} \rangle$ is generated by $\{s_i\}$). Since $s_j \in \LS{}$ if $j < i$ and $S(r) \not\in\LS{}$, we have $j = i$. Therefore we obtain $s_i = S(r)$ and $s_i \not\in \LS{}$.
\end{proof}

\section{Why do we need to compute divisions of S-polynomials?}\label{sec.S-poly}

As an application of Theorem \ref{thm.criterion_by_signature}, let us give a mathematical answer to the question ``Why do we need to compute remainders of divisions of S-polynomials to get \gr\ bases?''. As far as the author knows, all previous algorithms for computing \gr\ bases require computing remainders of divisions of S-polynomials by using division algorithms, Macaulay matrices and so on. Thus, several researchers have evaluated the computational complexity and presented improvements of these computations. It is well known that this method certainly produce a non-trivial leading monomial and is a part of the Buchberger's criterion. However, in the context of simply obtaining \gr\ bases, we still do not know if this method is really necessary.

From the previous section, we know that in order to get \gr\ bases we have to vanish the non-zero elements in $\Gobs{} = \langle \LSL{} \rangle / \langle \LS{} \rangle$. Let us focus on the minimum element in $\LSL{} \setminus \LS{}$. Then the remainder of a division of an S-polynomial appears naturally.

\begin{theorem}\label{thm.why_need_s-poly_division}
Assume that $F$ is not a \gr\ basis. Let $f$ be a non-zero element in $I$ such that $\LM(f) \not\in \langle \LM(F) \rangle$.
\begin{itemize}
	\item[(a)] The signature $S(f)$ is an element of $\LSL{} \setminus \LS{}$.
	\item[(b)] If the signature $S(f)$ is the minimum element of $\LSL{} \setminus \LS{}$ and $S(f) = \hat{S}(p)$ for an S-pair $p$, then it holds that $\LM(f) = \LM(r)$ and $S(f) = S(r)$, where $r$ is the remainder of any division of $\Spoly(p)$ with $F$.
	\item[(c)] In (b), the difference
	\[ \frac{1}{\LC(f)} f - \frac{1}{\LC(r)}r\]
	is $0$ or an element of signature smaller than $S(f)$. In particular, if the all terms of $f$ are not in $\langle \LM(F) \rangle$, then $f = cr$ for some $c \in K$.
\end{itemize}
\end{theorem}

\begin{proof}
For (a), if the signature $S(f)$ is not an element of $\LSL{}$, then it holds that $\LM\left(\overline{S(f)}\right) = \LM(f)$ from Lemma \ref{lem.Signature_differ_LM}. However, it contradicts to $\LM(f) \not \in \langle \LM(F) \rangle$. Since the signature of an element in $I \setminus \{0\}$ is not in $\LS{}$, the signature $S(f)$ is an element of $\LSL{} \setminus \LS{}$.

For (b) and (c), assume that the signature $S(f)$ is the minimum element of $\LSL{} \setminus \LS{}$ and $S(f) = \hat{S}(p)$ for an S-pair $p = (x^{\gamma}e_k, x^{\delta}e_{\ell})$. Take a division of the S-polynomial $\Spoly(p) = \frac{1}{\LC(f_k)} x^{\gamma}f_k - \frac{1}{\LC(f_{\ell})}x^{\delta}f_{\ell}$ with $F$:
\[ \Spoly(p) = \sum_{i=1}^m h_i f_i + r,\ \LM(h_if_i) < \LM(x^{\delta}f_{\ell}). \]
Put $u = \frac{1}{\LC(f_k)} x^{\gamma}e_k - \frac{1}{\LC(f_{\ell})}x^{\delta}e_{\ell} - \sum_{i=1}^m h_i e_i$. Then we have $\bar{u} = r$ and $\LM(u) = x^{\delta}e_{\ell} = S(f) \not\in \LS{}$ (Proposition \ref{prop.LMSyz_defines_signatures}). Therefore it holds that $r = \bar{u} \neq 0$ and $S(r) = S(f)$ from Lemma \ref{lem.signature_or_not}. Let $v$ be an element in $R^m$ such that $\bar{v} = f$ and $\LM(v) = S(f)$. Put
\[ w = \frac{1}{\LC(v)}v - \frac{1}{\LC(u)}u\ \text{and}\ g = \bar{w} =\frac{1}{\LC(v)}f - \frac{1}{\LC(u)}r. \]
If $g = 0$, then we obtain $f = \frac{\LC(v)}{\LC(u)} r$ and $\LM(f) = \LM(r)$. In particular, the difference in (c) is also $0$. If $g \neq 0$, then it holds that
\[ S(g)  \leq \LM(w) < \LM(u) = S(f).\]
Since the signature $S(f)$ is the minimum element of $\LSL{} \setminus \LS{}$, the signature $S(g)$ is not an element of $\LSL{}$. Therefore we have $\LM(g) = \LM\left(\overline{S(g)}\right) \in \langle \LM(F) \rangle$ (Lemma \ref{lem.Signature_differ_LM}). It implies that
\[ \frac{\LC(f)}{\LC(v)} \LM(f) = \frac{\LC(r)}{\LC(u)}\LM(r) \]
since those are not elements of $\langle \LM(F) \rangle$. In particular, we have
\[ \frac{1}{\LC(f)} f - \frac{1}{\LC(r)} r = \frac{\LC(u)}{\LC(v)\LC(r)} f - \frac{1}{\LC(r)}r =\frac{\LC(u)}{\LC(r)} g,\]
therefore the signature of the difference in (c) is smaller than $S(f)$.

Note that in general, an element $h \in I \setminus \{0\}$ of signature smaller than $\min \left( \LSL{} \setminus \LS{} \right)$ satisfies that $\LM(h) \in \langle \LM(F) \rangle$ from Lemma \ref{lem.Signature_differ_LM} again. Then the difference in (c) is $0$ if the all terms of $f$ (and $r$) are not in $\langle \LM(f) \rangle$.
\end{proof}

We say a tuple of polynomials $F$ is \textit{simplified} if for any $f \in F$, the all non-leading terms of $f$ are not in $\langle \LM(F) \rangle$. It is easy to make a simplified tuple $\tilde{F}$ such that $\LM(F) = \LM(\tilde{F})$ by taking reductions with $F$ over non-leading terms. We call such a tuple $\tilde{F}$ a \textit{simplification} of $F$. Note that common implementations of computing reduced \gr\ bases includes steps taking simplifications since any reduced \gr\ basis is simplified. Then assuming that given tuple of polynomials is simplified does not make the situation special.

We give an answer to the question ``Why do we need to compute remainders of divisions of S-polynomials to get \gr\ bases?'' for a homogeneous simplified polynomials $F$ and a graded term order $<$.

\begin{lemma}\label{lem.degree_of_sig}
Assume that $F$ consists of homogeneous elements and $<$ is graded. Then for any homogeneous element $f \in I \setminus \{0\}$, it holds that $\deg S(f) = \deg f$. Here we define the degree of $x^{\alpha}e_i$ as $\deg x^{\alpha}e_i = \deg x^{\alpha}f_i$.
\end{lemma}

\begin{proof}
Let $u$ be an element of $R^m$ such that $\bar{u} = f$ and $\LM(u) = S(f)$. Denote by $u_d$ the terms of $u$ of degree $d$. We have $u_d \in \Syz(F)$ for $d \neq \deg f$. Since $S(f) \not\in \LS{}$, we have $S(f) = \LM(u_{\deg f})$.
\end{proof}

\begin{theorem}\label{thm.minimum_still_minimum}
Assume that $F$ is not a \gr\ basis, $F$ consists of homogeneous elements and $<$ is graded. Let $s$ be the minimum element of $\LSL{}\setminus \LS{}$. Let $f$ be a non-zero homogeneous element in $I$ such that $\LM(f) \not\in \langle \LM(F) \rangle$. Put $F' = F \cup \{f\}$ and $f_{m+1} = f$. If $S(f_{m+1}) > s$, then $s$ is the minimum element of $\LSL[\mathit{F}']\setminus \LS[\mathit{F}']$.
\end{theorem}

\begin{proof}
First we show that $s \in \LSL[\mathit{F}']\setminus \LS[\mathit{F}']$. Since $\LSL{} \subset \LSL[\mathit{F}']$, it is clear that $s \in \LSL[\mathit{F}']$. If $s \in \LS[\mathit{F}']$, then there exist a homogeneous element $u \in R^m = Re_1 \oplus \cdots \oplus Re_m$ and a homogeneous element $h \in R$ such that $s = \LM(u+he_{m+1})$ and $u+he_{m+1}$ is a homogeneous element in $\Syz(F')$. Since $s \not\in \LS{}$ and $f_{m+1} \neq 0$, we have $h\neq 0$ and $u \neq 0$. Moreoreve, since $s \in R^m$, we have $s = \LM(u) \succ \LM(he_{m+1})$. Indeed, if $\LM(u) \prec \LM(he_{m+1})$, then $s = \LM(he_{m+1}) \in Re_{m+1}$. However, it is a contradiction to $R^m \cap Re_{m+1} = \{0\}$. From Lemma \ref{lem.signature_or_not}, it and the equality $hf_{m+1} = -\bar{u}$ implies that $s = S(hf_{m+1})$. Therefore it holds that
\[ \deg s = \deg hf_{m+1} \geq \deg f_{m+1} = \deg S(f_{m+1}). \]
Since $S(f_{m+1}) > s$ and $<$ is graded, the equality $\deg s = \deg S(f_{m+1})$ holds. Then we have $h \in K$. However, it implies that $s = S(f_{m+1})$ and it is a contradiction to $S(f_{m+1}) > s$. 

Next we show that $s$ is minimum in $\LSL[\mathit{F}']\setminus \LS[\mathit{F}']$. If there exists an S-pair $p = (x^{\gamma}e_k, x^{\delta}e_{\ell})\ (1\leq k <\ell \leq m+1)$ of $F'$ such that $x^{\delta}e_{\ell} \in \LSL[\mathit{F}']\setminus \LS[\mathit{F}']$ and $x^{\delta}e_{\ell} < s$, then $\ell = m+1$ since $s$ is minimum in $\LSL{}\setminus \LS{}$ and $\LS{} \subset \LS[\mathit{F}']$. Therefore we have
\[ \deg s \geq  \deg x^{\delta}f_{m+1} \geq  \deg f_{m+1} = \deg S(f_{m+1}).\]
Since $S(f_{m+1}) > s$, the equalities
\[ \deg s = \deg x^{\delta}f_{m+1} = \deg f_{m+1}\]
hold. Then we have $x^{\delta} = 1$. However, by definition of S-pairs, it implies that $\LM(f_{m+1}) = x^{\gamma} \LM(f_k) \in \langle \LM(F) \rangle$, and it is a contradiction.
\end{proof}

\begin{corollary}\label{cor.S-poly_is_needed}
Assume that $F$ is not a \gr\ basis, $F$ consists of homogeneous elements and $<$ is graded. Let $G$ be a homogeneous \gr\ basis of $I$ including $F$. Then there exist a subset $F'$ of $G$ including $F$, an element $g \in G \setminus F'$ and an S-pair $p$ of $F$ such that
\begin{itemize}

\item the guessed signature $\hat{S}(p)$ is the minimum element of $\LSL{} \setminus \LS{}$,
\item $\LM(g) = \LM(r)$ and $S_{F'}(g)  = S_{F'}(r) = \hat{S}(p)$, where $S_{F'}(g)$ is the signature of $g$ with respect to $F'$, and $r$ is the remainder of any division of an S-polynomial $\Spoly(p)$ with $F'$,
\item if $G$ is simplified, then there exists $c \in K$ such that $g = cr$.
\end{itemize}
\end{corollary}

\begin{proof}
Let $s$ be the minimum element of $\LSL{}\setminus \LS{}$. Let $p = (x^{\gamma}e_k,x^{\delta}e_{\ell})\ (1 \leq k < \ell \leq m)$ be an S-pair such that $\hat{S}(p) = s$. Put $F_m = F$. Pick an element $f_{m+1} \in G \setminus F_m$ such that $\LM(f_{m+1}) \not\in \langle \LM(F_m) \rangle$. Put $F_{m+1} = F \cup \{f_{m+1}\}$. Let $S_{F_m}(f_{m+1})$ be the signature of $f_{m+1}$ with respect to $F_m$. If $S_{F_m}(f_{m+1}) > s$, then $s$ is the minimum element of $\LSL[\mathit{F}_{m+1}]\setminus \LS[\mathit{F}_{m+1}]$. Therefore $F_{m+1}$ is not a \gr\ basis. Repeat this process until it picks an element $f_{m+k} \in G \setminus F_{m+k-1}$ such that $\LM(f_{m+k}) \not\in \langle \LM(F_{m+k-1}) \rangle$ and
\[ s = S_{F_{m+k-1}}(f_{m+k}) = \min \left( \LSL[\mathit{F}_{m+k-1}]\setminus \LS[\mathit{F}_{m+k-1}] \right) = x^{\delta}f_{\ell}.\]
Let $r$ be the remainder of any division of the S-polynomial
\[ \Spoly(p) = \frac{1}{\LC(f_k)}x^{\gamma} f_k - \frac{1}{\LC(f_{\ell})}x^{\delta} f_{\ell} \]
with $F_{m+k-1}$. Then, from Theorem \ref{thm.why_need_s-poly_division}, we have $\LM(f_{m+k}) = \LM(r)$ and $S_{F_{m+k-1}}(f_{m+k}) = S_{F_{m+k-1}}(r) = s$. Moreover, if $G$ is simplified, then the all terms of $f_{m+k}$ are not in $\langle \LM(F_{m+k-1}) \rangle  \subset \langle \LM(G) \rangle$, therefore we have $f_{m+k} = cr$, where $c = \LC(f_{m+k})/\LC(r)$.
\end{proof}

\section{Examples of transitions of $\Gobs{}$ in a signature based algorithm}\label{sec.example}

Let us look at computational examples of $\Gobs{}$. We use a naive implementation of a signature based algorithm (Algorithm \ref{alg:minimum_SBA}), which is similar to the algorithms presented in \cite{AP11,VY17,Sak20}. The difference of Algorithm \ref{alg:minimum_SBA} is that it iterates to update the tuple of generators $F$, and then the signatures change for each step. The performance is not discussed here. The termination is clear since $R$ is a Noether ring. We use SageMath\cite{sagemath} to implement and run Algorithm \ref{alg:minimum_SBA}.

\begin{figure}[!t]
\begin{algorithm}[H]
    \caption{Signature based algorithm}
    \label{alg:minimum_SBA}
    \begin{algorithmic}[1]
    \REQUIRE a tuple $F = (f_1,f_2,\ldots, f_m)$ of elements in $R$
    \ENSURE a \gr\ basis of $I = \langle f_1, f_2 ,\ldots, f_m \rangle$
    \STATE $S \leftarrow \emptyset$
    \WHILE{$S = \emptyset$}
    \STATE $D \leftarrow \{ \hat{S}(p) \mid \text{$p$ is a standard S-pair of $F$}\}$
    \STATE sort $D$ and pick standard S-pairs, let $D = \{ \hat{S}(p_1), \ldots, \hat{S}(p_d)\}$ and $\hat{S}(p_1) \prec \hat{S}(p_2) \prec \cdots \hat{S}(p_d)$    
    \FOR{i = 1,2,\ldots,d}
    \STATE $r \leftarrow \text{the remainder of any division of $\Spoly(p_i)$ with $F$}$
    \IF{$r = 0$}
    \STATE $S \leftarrow S \cup \{\hat{S}(p_i) \}$
    \ENDIF
    \IF{$r \neq 0$}
    \STATE $F \leftarrow F \cup \{r\}$, $S \leftarrow \emptyset$
    \STATE break this loop
    \ENDIF
    \ENDFOR
    \IF{S = D}
    \RETURN $F$
    \ENDIF
    \ENDWHILE
    \end{algorithmic}
\end{algorithm}
\end{figure}

\begin{example}\label{ex.not_decrease}
Let $R = \mathbb{Q}[x,y,z]$ be the polynomial ring equipped with the graded lexicographic order of $x > y > z$. Let
\[ \begin{aligned}
 f_1 &= x^3 y - z,\\
 f_2 &= xyz - 2y,\\
 f_3 &= xy^2 - z^2.
 \end{aligned}\]
Using Algorithm \ref{alg:minimum_SBA}, we get a sequence of tuples $F_3, F_4, \ldots, F_{11}$ such that
\begin{itemize}
 \item $F_j = (f_1,f_2,\ldots,f_j)$, $\LM(f_j) \not\in \langle \LM(f_1),\LM(f_2), \ldots, \LM(f_{j-1}) \rangle$,
 \item the signature of $f_{j+1}$ with respect to $F_j$ is the minimum element of $\LSL[\mathit{F}_{\mathit{j}}] \setminus \LS[\mathit{F}_{\mathit{j}}]$ and
 \item $F_{11}$ is a \gr\ basis of $I = \langle f_1,f_2,f_3 \rangle$.
\end{itemize}
Let us observe transition of $\Gobs[\mathit{F}_{\mathit{i}}]$. The following are minimal free resolutions of $\Gobs[\mathit{F}_{\mathit{i}}]$ computed by sage math packages, and we also compare the monomial ideals generated by $\LM(F_i)$. The generator of each monomial ideal wrote in the last is the new leading monomial $\LM(f_{i})$ added in that step.
\[ \begin{aligned}
&\Gobs[\mathit{F}_{3}] \leftarrow R^3 \leftarrow R^6 \leftarrow R^3 \leftarrow 0, &\langle \LM(F_3) \rangle &= \left\langle \begin{aligned}&x^3y,xyz,\\
&xy^2 \end{aligned} \right\rangle,\\
&\Gobs[\mathit{F}_4] \leftarrow R^2 \leftarrow R^5 \leftarrow R^3 \leftarrow 0, &\langle \LM(F_4) \rangle &= \left\langle \begin{aligned} &x^3y,xyz,\\
&xy^2,z^3
\end{aligned}\right\rangle,\\
&\Gobs[\mathit{F}_{5}] \leftarrow R^4 \leftarrow R^8 \leftarrow R^4 \leftarrow 0, &\langle \LM(F_5) \rangle &= \left\langle \begin{aligned}&xyz,xy^2\\
&z^3,x^2y \end{aligned}\right\rangle,\\
&\Gobs[\mathit{F}_{6}] \leftarrow R^5 \leftarrow R^{11} \leftarrow R^7 \leftarrow R^1 \leftarrow 0, &\langle \LM(F_6) \rangle &= \langle xy,z^3 \rangle,\\
&\Gobs[\mathit{F}_{7}] \leftarrow R^4 \leftarrow R^9 \leftarrow R^6 \leftarrow R^1 \leftarrow 0, &\langle \LM(F_7) \rangle &= \langle xy,z^3,y^2z \rangle,\\
&\Gobs[\mathit{F}_{8}] \leftarrow R^2 \leftarrow R^4 \leftarrow R^2 \leftarrow 0,&\langle \LM(F_8) \rangle &= \left\langle \begin{aligned}
&xy,z^3,\\
&y^2z,y^3 \end{aligned} \right\rangle,\\
&\Gobs[\mathit{F}_{9}] \leftarrow R^1 \leftarrow R^2 \leftarrow R^1 \leftarrow 0,&\langle \LM(F_9) \rangle &= \left\langle \begin{aligned}
&xy,z^3,y^2z,\\
&y^3,xz^2 \end{aligned} \right\rangle,\\
&\Gobs[\mathit{F}_{10}] \leftarrow R^1 \leftarrow R^2 \leftarrow R^1 \leftarrow 0, &\langle \LM(F_{10}) \rangle &= \left\langle \begin{aligned} &xy,z^3,y^2z,\\
&y^3,xz^2,yz^2 \end{aligned} \right\rangle,\\
&\Gobs[\mathit{F}_{11}] \leftarrow 0, &\langle \LM(F_{11}) \rangle &= \left\langle \begin{aligned}
&xy,z^3,y^2z,\\
&y^3,yz^2,xz \end{aligned} \right\rangle.
\end{aligned}\]
For more detail, see Appendix.
\end{example}

\begin{example}\label{ex.example2}

Let us change the term order from Example \ref{ex.not_decrease}. If we set the lexicographic order of $x>y>z$ on $R$, then we get a sequence of tuples $F_3, F_4, \ldots, F_{13}$ with the same conditions as in Example \ref{ex.not_decrease}. Let us observe the transition of $\Gobs[\mathit{F}_{\mathit{i}}]$:
\[ \begin{aligned}
&\Gobs[\mathit{F}_{3}] \leftarrow R^3 \leftarrow R^6 \leftarrow R^3 \leftarrow 0, &\langle \LM(F_3) \rangle &= \left\langle \begin{aligned}&x^3y,xy^2,\\
&xyz \end{aligned} \right\rangle,\\
&\Gobs[\mathit{F}_4] \leftarrow R^3 \leftarrow R^6 \leftarrow R^3 \leftarrow 0, &\langle \LM(F_4) \rangle &= \left\langle \begin{aligned} &x^3y,xyz,\\
&y^2
\end{aligned}\right\rangle,\\
&\Gobs[\mathit{F}_{5}] \leftarrow R^2 \leftarrow R^5 \leftarrow R^3 \leftarrow 0, &\langle \LM(F_5) \rangle &= \left\langle \begin{aligned}&x^3y,xyz\\
&y^2,xz^3 \end{aligned}\right\rangle,\\
&\Gobs[\mathit{F}_{6}] \leftarrow R^4 \leftarrow R^8 \leftarrow R^4 \leftarrow 0, &\langle \LM(F_6) \rangle &= \left\langle \begin{aligned}
&xyz,y^2,\\
&xz^3,x^2y
\end{aligned}\right\rangle,\\
&\Gobs[\mathit{F}_{7}] \leftarrow R^5 \leftarrow R^{11} \leftarrow R^7 \leftarrow R^1 \leftarrow 0, &\langle \LM(F_7) \rangle &= \left\langle \begin{aligned}
&y^2,xz^3,\\
&xy
\end{aligned}\right\rangle,\\
&\Gobs[\mathit{F}_{8}] \leftarrow R^5 \leftarrow R^{12} \leftarrow R^9 \leftarrow R^2 \leftarrow 0, &\langle \LM(F_8) \rangle &= \left\langle
xz^3,y \right\rangle,\\
&\Gobs[\mathit{F}_{9}] \leftarrow R^2 \leftarrow R^4 \leftarrow R^2 \leftarrow 0,&\langle \LM(F_9) \rangle &= \left\langle \begin{aligned}
&xz^3,y,\\
&z^8 \end{aligned} \right\rangle,\\
&\Gobs[\mathit{F}_{10}] \leftarrow R^2 \leftarrow R^4 \leftarrow R^2 \leftarrow 0, &\langle \LM(F_{10}) \rangle &= \left\langle \begin{aligned}
&xz^3,y,\\
&z^7 \end{aligned} \right\rangle,\\
&\Gobs[\mathit{F}_{11}] \leftarrow R^1 \leftarrow R^2 \leftarrow R^1 \leftarrow 0, &\langle \LM(F_{11}) \rangle &= \left\langle \begin{aligned}
&y,z^7,\\
&xz^2 \end{aligned} \right\rangle,\\
&\Gobs[\mathit{F}_{12}] \leftarrow R^1 \leftarrow R^2 \leftarrow R^1 \leftarrow 0, &\langle \LM(F_{12}) \rangle &= \left\langle \begin{aligned}
&y,xz^2\\
&z^6 \end{aligned} \right\rangle,\\
&\Gobs[\mathit{F}_{13}] \leftarrow 0, &\langle \LM(F_{13}) \rangle &= \left\langle \begin{aligned}
&y,z^6\\
&xz \end{aligned} \right\rangle.
\end{aligned}\]
\end{example}

\begin{example}\label{ex.example3}
Let us see an interesting example. Let $R = \mathbb{Q}[x,y,z,w]$ be the polynomial ring equipped with the lexicographic order of $x>y>z>w$. Let
\[ \begin{aligned}
 f_1 &= xy +3y^2-2zw,\\
 f_2 &= 2x^2+y^2-5w^2,\\
 f_3 &= zw + \frac{3}{2} w^2
 \end{aligned}\]
and put $F = (f_1,f_2,f_3)$. Then we have
\[ \Gobs{} = \langle y e_2 \rangle/\langle xye_2, yzwe_2, x^2e_3,xye_3 \rangle.\]
Therefore one may consider that we obtain a \gr\ basis of $I = \langle f_1,f_2,f_3 \rangle$ by only one step that reduces the S-polynomial $\Spoly(xe_1,ye_2) = xf_1 - \frac{1}{2}yf_2$. From Theorem \ref{thm.why_need_s-poly_division}, the leading monomial of the remainder of any division of $\Spoly(xe_1,ye_2)$ with $F$ is constant and it is $xw^2$ by computing a division of $\Spoly(xe_1,ye_2)$. However, in fact, we have $\langle xy, x^2, zw, xw^2 \rangle \subsetneq \ini_{<}(I) = \langle x^2,xy,xw^2,y^4,y^3z,zw \rangle$, therefore we do not obtain a \gr\ basis of $I$ by eliminating the minimum guessed signature $ye_2$. Moreover, the first Betti number of the module of \gr ness obstructions increase.

On the other hand, if we set the degree reversed lexicographic order of $x>y>z>w$ on $R$, then we obtain a \gr\ basis of $I$ by only one step that reduces $\Spoly(xe_1,ye_2)$. Let us observe the transitions of $\Gobs[\mathit{F}_{\mathit{i}}]$ in these two cases. For the lexicographic order:
\[ \begin{aligned}
&\Gobs[\mathit{F}_{3}] \leftarrow R^1 \leftarrow R^2 \leftarrow R^1 \leftarrow 0, &\langle \LM(F_3) \rangle &= \left\langle x^2,xy,zw \right\rangle,\\
&\Gobs[\mathit{F}_4] \leftarrow R^2 \leftarrow R^4 \leftarrow R^2 \leftarrow 0, &\langle \LM(F_4) \rangle &= \left\langle \begin{aligned} &x^2,xy,zw,\\
&xw^2
\end{aligned}\right\rangle,\\
&\Gobs[\mathit{F}_{5}] \leftarrow R^1 \leftarrow R^2 \leftarrow R^1 \leftarrow 0, &\langle \LM(F_5) \rangle &= \left\langle \begin{aligned}&x^2,xy,zw,\\
&xw^2,y^3z \end{aligned}\right\rangle,\\
&\Gobs[\mathit{F}_{6}] \leftarrow 0, &\langle \LM(F_6) \rangle &= \left\langle \begin{aligned}
&x^2,xy,zw,\\
&xw^2,y^3z,y^4
\end{aligned}\right\rangle.
\end{aligned}\]
For the graded reverse lexicographic order:
\[ \begin{aligned}
&\Gobs[\mathit{F}_{3}] \leftarrow R^1 \leftarrow R^2 \leftarrow R^1 \leftarrow 0, &\langle \LM(F_3) \rangle &= \left\langle x^2,xy,zw \right\rangle,\\
&\Gobs[\mathit{F}_4] \leftarrow 0, &\langle \LM(F_4) \rangle &= \left\langle \begin{aligned} &x^2,xy,zw,\\
&y^3
\end{aligned}\right\rangle.
\end{aligned}\]
\end{example}

\begin{example}\label{ex.finite_field}
Let us consider the case of coefficients in a finite field. Let $R = \mathbb{Z}/5\mathbb{Z}[x,y,z]$ be the polynomial ring equipped with the degree lexicographic order of $x>y>z$. Let
\[ \begin{aligned}
 f_1 &= xy +4z +2,\\
 f_2 &= xyz+y^2+1,\\
 f_3 &= x^2y +4z^2
 \end{aligned}\]
Then we get a sequence of tuples $F_3,F_4,\ldots, F_9$ with the same conditions as in Example \ref{ex.not_decrease}. Let us see the transition of $\Gobs[\mathit{F}_{\mathit{i}}]$:
\[ \begin{aligned}
&\Gobs[\mathit{F}_{3}] \leftarrow R^2 \leftarrow R^4 \leftarrow R^2 \leftarrow 0, &\langle \LM(F_3) \rangle &= \left\langle xy \right\rangle,\\
&\Gobs[\mathit{F}_4] \leftarrow R^2 \leftarrow R^4 \leftarrow R^2 \leftarrow 0, &\langle \LM(F_4) \rangle &= \left\langle xy, y^2 \right\rangle,\\
&\Gobs[\mathit{F}_{5}] \leftarrow R^1 \leftarrow R^3 \leftarrow R^2 \leftarrow 0, &\langle \LM(F_5) \rangle &= \left\langle xy, y^2,xz^2 \right\rangle,\\
&\Gobs[\mathit{F}_{6}] \leftarrow R^2 \leftarrow R^4 \leftarrow R^2 \leftarrow 0, &\langle \LM(F_6) \rangle &= \left\langle xy, y^2,xz \right\rangle,\\
&\Gobs[\mathit{F}_{7}] \leftarrow R^2 \leftarrow R^5 \leftarrow R^4 \leftarrow R^1 \leftarrow 0, &\langle \LM(F_7) \rangle &= \left\langle \begin{aligned}
&xy, y^2,xz,\\
&z^3
\end{aligned} \right\rangle,\\
&\Gobs[\mathit{F}_{8}] \leftarrow R^1 \leftarrow R^2 \leftarrow R^1 \leftarrow 0, &\langle \LM(F_8) \rangle &= \left\langle \begin{aligned}
&xy, y^2, xz,\\
&z^3,yz^2
\end{aligned} \right\rangle,\\
&\Gobs[\mathit{F}_{9}] \leftarrow 0, &\langle \LM(F_9) \rangle &= \left\langle \begin{aligned}
&xy, y^2,xz,\\
&z^3,yz^2,x^2
\end{aligned} \right\rangle.\\
\end{aligned}\]
\end{example}

From the above examples, the sequence of $\Gobs[\mathit{F}_{\mathit{i}}]$ does not monotonically go to the zero-module in general. Moreover, the sequence of Betti numbers or projective dimensions of $\Gobs[\mathit{F}_{\mathit{i}}]$ also does not monotonically go to $0$. Here one may suggest the following question.
\begin{question}
Does there exists an algorithm such that the values of some invariant of $\Gobs[\mathit{F}_{\mathit{i}}]$ monotonically go to $0$? Is it fast?
\end{question}
It seems that the increase and decrease of the first Betti numbers link to phases of the leading monomials (Example \ref{ex.not_decrease}, Example \ref{ex.example2}, Example \ref{ex.finite_field}). However, there is an exceptional example (Example \ref{ex.example3}). We have not yet obtained consideration of it in this paper.

\section{\gr\ degenerations and signatures}\label{sec.flatness_of_family}

In fact, there exists an affine scheme $X$ in $\mathbb{A}^n_K \times_{K} \mathbb{A}^1_K$ such that the projection $\pi: X \rightarrow \mathbb{A}^1_K$ is flat, generic fibers $X_t = \pi^{-1}(t)$ over $t \neq 0$ are isomorphic to the affine scheme $\spec R/I$, and the special fiber $X_0 = \pi^{-1}(0)$ is isomorphic to the affine scheme $\spec R/(\ini_{<} I)$ \cite{Bay82,Eis95}. Such affine schemes are called \textit{\gr\ degenerations} of $\spec R/I$. We recall how to construct a \gr\ degeneration from a \gr\ basis and a weighting vector of positive integers.

\begin{definition}\label{def.compatible_weight}
Let $A$ be a finite set of monomials. In fact, there exists a vector of positive integers $\omega \in \mathbb{Z}^{n}_{>0}$ such that for any monomials $x^{\alpha}, x^{\beta} \in A$, $x^{\alpha} < x^{\beta}$ if and only if $\omega \cdot \alpha < \omega \cdot \beta$ \cite{Rob85}. Here we denote by $\omega \cdot \alpha$ the ordinal inner product of $\omega$ and $\alpha$. We say that such vector $\omega$ is \textit{compatible} with $A$.
\end{definition}

\begin{definition}\label{def.omega-degrees}
Assume that a vector of positive integers $\omega \in \mathbb{Z}^n_{>0}$ is compatible with the set of monomials appeared in elements of $F$. We define the \textit{$\omega$-degree of a monomial $x^{\alpha}$} as $\deg_{\omega} x^{\alpha} = \omega \cdot \alpha$. Also for any element $f \in I$, we define the \textit{$\omega$-degree of a polynomial $f$} as $\deg_{\omega} f = \max \{ \deg_{\omega} x^{\alpha} \mid \text{$x^{\alpha}$ appears in $f$} \}$. We denote by $\LW_{\omega} f$ the sum of all terms of $f$ of $\omega$-degree $\deg_{\omega} f$, and call $\LW_{\omega} f$ the \textit{top terms} of $f$ with respect to $\omega$.
\end{definition}

Let $f = \sum_{\alpha} c_{\alpha} x^{\alpha}$ be an element of $R$. We define notations
\[ f^t = \sum_{\alpha} c_{\alpha} t^{-\omega \cdot \alpha} x^{\alpha}\]
and
\[ f^{(t)} = t^{\deg_{\omega} f} f^t \]
for new variable $t$ independent to $x_1,\ldots,x_n$. The former is an element of the  Laurent polynomials ring $R[t,t^{-1}] = R \otimes_{K} K[t,t^{-1}]$, the latter is an element of the polynomial ring $R[t] = R \otimes_K K[t]$. Moreover, the latter is a homogeneous element of $R[t]$ with respect to the grading $\deg_{\omega} t^d x^{\alpha} = d + \omega \cdot \alpha$, we have $(f^{(t)})_{|t=0} = \LW_{\omega}(f)$.

In below, we fix the setting of Definition \ref{def.omega-degrees} and assume that all elements of $F$ are monic (namely, $\LC(f_i) = 1$). Therefore we have $\LW_{\omega} f_i = (f_i^{(t)})_{|t=0} = \LM(f_i)$. We denote $F^{(t)}_{\omega} = \{ f_i^{(t)} \mid i = 1,\ldots,a\}$.

\begin{theorem}[{\cite[15.8]{Eis95}}]\label{thm.Grobner_deformation}
Consider a family $X = \spec R[t]/\langle F^{(t)}_{\omega} \rangle$ on $\A^1_{K,t} = \spec K[t]$. The fibers $X_t$ over $t \neq 0$ are isomorphic to $\spec R/I$. Moreover, if $F$ is a \gr\ basis, this family is flat over $\A^1_{K,t} = \spec K[t]$ and the special fiber at $t = 0$ is isomorphic to $\spec R/(\ini_{<}I)$.
\end{theorem}

Our goal in this section is to give necessary and sufficient conditions of that the family $X = \spec R[t]/\langle F^{(t)}_{\omega} \rangle$ is flat over $\A^1_{K,t}$ from the point of view of signatures.

Let us start from analysis the flatness of $X = \spec R[t]/\langle F^{(t)}_{\omega} \rangle$. In the following discussion, we identify the $K[t]$-module $K[t]/\langle t \rangle$ as $K$. Artin gives a criterion for the flatness of the family $X$ via the syzygy modules.

\begin{theorem}[{\cite[1.3]{Art76lec}}, see also \cite{Bay93can}]\label{thm.artin}
The family $\spec R[t]/\langle F^{(t)}_{\omega} \rangle$ is flat over $\mathbb{A}^1_{K,t}$ if and only if the morphism
\[ \varphi : \Syz(F^{(t)}_{\omega}) \otimes_{K[t]} K \rightarrow \Syz(\LM(F))\]
\[ \begin{aligned}
e_i^{(t)} &\mapsto e_i\\
 t &\mapsto 0
\end{aligned}
\]
is surjective.
\end{theorem}

Considering initial modules in $R^m$, we obtain the following corollary of Theorem \ref{thm.artin} that states a relationship between the flatness and guessed signatures.

\begin{corollary}\label{cor.criterion_of_flatness}
The family $\spec R[t]/\langle F^{(t)}_{\omega} \rangle$ is flat over $\mathbb{A}^1_{K,t}$ if and only if $\LM\left(\varphi\left(\Syz(F^{(t)}_{\omega}) \otimes_{K[t]} K\right)\right) = \LSL{}$.
\end{corollary}

We denote by $\LImS{}$ the set of leading monomials of the image of the morphism $\varphi: \Syz(F^{(t)}_{\omega}) \otimes_{K[t]} K \rightarrow \Syz(\LM(F))$, namely, $\LImS{} = \LM\left(\varphi\left(\Syz(F^{(t)}_{\omega}) \otimes_{K[t]} K\right)\right)$. Combining this with the results we proved, we obtain the following theorem.

\begin{theorem}\label{thm.criterion_by_the_image_of_SyzFt}
A tuple $F$ is a \gr\ basis if and only if $\spec R[t]/\langle F^{(t)}_{\omega} \rangle$ is flat over $\mathbb{A}^1_{K,t}$ and $\LS{} = \LImS{}$.
\end{theorem}

\begin{proof}
If $F$ is a \gr\ basis, then by Theorem \ref{thm.criterion_by_signature}, Theorem \ref{thm.Grobner_deformation} and Corollary \ref{cor.criterion_of_flatness} we have that $\spec R[t]/\langle F^{(t)}_{\omega} \rangle$ is flat over $\mathbb{A}^1_{K,t}$ and $\LS{} = \LSL{} = \LImS{}$.

Conversely, assume that $\spec R[t]/\langle F^{(t)}_{\omega} \rangle$ is flat over $\mathbb{A}^1_{K,t}$ and $\LS{} = \LImS{}$. Then we have $\LImS{} = \LSL{}$ (Corollary \ref{cor.criterion_of_flatness}). Therefore $F$ is a \gr\ basis since $\LS{} = \LSL{}$ (Theorem \ref{thm.criterion_by_signature}).
\end{proof}

Assuming a special assumption on the weight vector $\omega$, we show that the set $\LS{}$ is included in $\LImS{}$.

\begin{lemma}\label{lem.LM_LW_Syz}
Let $V_F = \{v_1,\ldots, v_b\}$ be a \gr\ basis of the syzygy module $\Syz(F)$. Let $A$ be the sum of the following sets of monomials in $R$:
\begin{itemize}
	\item $\{ x^{\alpha} \mid \text{$x^{\alpha}$ appears in an element of $F$}\}$,
	\item $\{x^{\alpha} \LM(f_i) \mid \text{$x^{\alpha} e_i$ appears in an element of $V_F$}\}$.
\end{itemize}
Assume that $\omega$ is compatible with $A$. Then for any element $v$ of $V_F$, it holds that $\LM(\LW_{\omega}(v)) = \LM(v)$.
\end{lemma}

\begin{proof}
Assume that $v = \sum_{\alpha,i} c_{\alpha,i} x^{\alpha} e_i \in V_F$. By assumption, for any pair $(\alpha, i), (\beta, j)$ with $c_{\alpha,i}, c_{\beta,j} \neq 0$, $x^{\alpha} \LM(f_i) \prec x^{\beta} \LM(f_j)$ if and only if $\deg_{\omega} x^{\alpha} e_i < \deg_{\omega} x^{\beta} e_j$. Put $x^{\xi} = \LM\left(\overline{\LM(v)}\right)$. Then we have $\deg_{\omega} v = \deg_{\omega} \LM(v) = \deg_{\omega} x^{\xi}$, and for any term $x^{\alpha} e_i$ of $v$, $x^{\alpha} \LM(f_i) = x^{\xi}$ if and only if $\deg_{\omega} x^{\alpha} e_i = \deg_{\omega} v$. Therefore we have $\LM\left(\LW_{\omega}(v)\right) = \LM(v)$.
\end{proof}

\begin{lemma}\label{lem.LSF_subset_LImS}
Set the same assumption of Lemma \ref{lem.LM_LW_Syz}. We have
\[ \langle \LW_{\omega}(u) \mid u \in \Syz(F) \rangle = \im \varphi \]  
In particular, it holds that $\LS{} \subset \LImS{}$.
\end{lemma}

\begin{proof}

If the set equation holds, then it implies that $\LS{} \subset \LImS{}$ since $\LM(v) = \LM(\LW_{\omega}(v))$ for any element $v \in V_F$ (Lemma \ref{lem.LM_LW_Syz}).

For any $v \in \Syz(F^{(t)}_{\omega})$, let us compute $\varphi(v \otimes 1)$. Denote
\[ v = \sum_{D} v_D = \sum_{D} \sum_{D = \deg_{\omega}(x^{\alpha}f_i) + d} c_{\alpha, i, d} t^d x^{\alpha} e_i^{(t)}.\]
Since $f_i^{(t)}$ is a homogeneous element in $R[t]$, we have $v_D \in \Syz(F^{(t)}_{\omega})$. Then it holds that
\[ \varphi(v \otimes 1) = \sum_{D} \varphi(v_D \otimes 1) = \sum_{D} \sum_{D = \deg_{\omega}(x^{\alpha}f_i)} c_{\alpha, i, 0} x^{\alpha} e_i.\]
Put
\[ u_D = \sum_{D = \deg_{\omega}(x^{\alpha}f_i) + d} c_{\alpha, i, d} x^{\alpha} e_i. \]
Then we have $ \varphi(v_D \otimes 1)=\LW_{\omega} (u_D) $ if $\varphi(v_D \otimes 1) \neq 0$. Since $v_D \in \Syz(F_{\omega}^{(t)})$ and $\left(f_i^{(t)}\right)_{|t=1} = f_i$, this element $u_D$ is a syzygy of $F$. Therefore we have $ \im \varphi \subset \langle \LW_{\omega}(u) \mid u \in \Syz(F) \rangle$.

Conversely, Let $u = \sum_{\alpha,i} c_{\alpha,i} x^{\alpha} e_i$ be a non-zero syzygy of $F$ and $D_0 = \deg_{\omega} u$. Taking homogenization of the equation
\[ \sum_{\alpha,i} c_{\alpha,i} x^{\alpha}f_i = 0,\]
we get the following syzygy of $F_{\omega}^{(t)}$:
\[ v = \sum_{\alpha,i} c_{\alpha,i}t^{D_0 - \deg_{\omega} (x^{\alpha}f_i)} x^{\alpha}e_i^{(t)}. \]
Therefore we have $\varphi(v \otimes 1) = \sum_{D_0 = \deg_{\omega}(x^{\alpha}f_i)} c_{\alpha,i} x^{\alpha} e_i = \LW_{\omega}(u)$.
\end{proof}

From Lemma \ref{lem.LSF_subset_LImS}, the module of \gr ness obstructions $\Gobs{}$ is divided into the direct sum $\Gobs{} = M(F) \oplus N(F)$, where
\[ M(F) = \langle \LImS{} \rangle/\langle \LS{} \rangle,\ N(F) = \langle \LSL{} \rangle /\langle \LImS{} \rangle.\]
Our results in this paper are represented by these summands as follows.
\begin{itemize}
	\item (Flatness) The family $\spec R[t]/\langle F^{(t)}_{\omega} \rangle$ is flat over $\mathbb{A}^1_{K,t}$ if and only if $N(F) = 0$.
	\item (\gr ness) A tuple $F$ is a \gr\ basis if and only if $M(F) = N(F) = 0$
\end{itemize}

\begin{acknowledgments}
The author gratefully acknowledges the support of past and present members of my institution Mitsubishi Electric Corporation, Information Technology R\&D Center. The author wishes to thank Kazuhiro Yokoyama and Yuki Ishihara for comments on an earlier version of this paper.
\end{acknowledgments}

\bibliographystyle{alpha}
\bibliography{refloc}

\section*{Appendix}\label{sec.append}
Let us see more detail of Example \ref{ex.not_decrease}. The following are explicit choice of polynomials and structures of $\Gobs[\mathit{F}_{\mathit{i}}]$. We compute a \gr\ basis of the syzygy modules $\Syz(F_i)$ independent from Algorithm \ref{alg:minimum_SBA} to determine the structure of $\Gobs[\mathit{F}_{\mathit{i}}]$. The notation $S_F(f)$ denotes the signature of $f$ with respect to $F$.\\[2ex]

\noindent\underline{Input}
\[ \begin{aligned}
 f_1 &= x^3 y - z,\\
 f_2 &= xyz - 2y,\\
 f_3 &= xy^2 - z^2.
 \end{aligned}\]

\noindent\underline{Get $f_4$}
\begin{itemize}
	\item (Tuple) $F_3 = (f_1,f_2,f_3)$,
	\item (\gr\ obstructions)
	\[\Gobs[\mathit{F}_{3}] = \langle x^2 e_2, x^2 e_3, z e_3 \rangle/\langle x^2z^2e_2, x^3ze_2,x^3ye_2,x^3ye_3,xyze_3 \rangle,\]
	\item (Minimal free resolution)
	\[0 \leftarrow \Gobs[\mathit{F}_{3}] \leftarrow R^3 \leftarrow R^6 \leftarrow R^3 \leftarrow 0,\]
	\item (Choice of polynomial with minimum signature) $f_4 = z^3-2y^2$ as the remainder of a division of $\Spoly(ye_2,ze_3)$ and $S_{F_3}(f_4) = ze_3$.
\end{itemize}

\noindent\underline{Get $f_5$}
\begin{itemize}
	\item (Tuple) $F_4 = (f_1,f_2,f_3,f_4)$,
	\item (\gr\ obstructions)
	\[\Gobs[\mathit{F}_{4}] = \langle x^2 e_2, x^2 e_3 \rangle/\langle x^2z^2e_2, x^3ze_2,x^3ye_2,ze_3,x^3ye_3,xye_4\rangle,\]
	\item (Minimal free resolution)
	\[0 \leftarrow \Gobs[\mathit{F}_{4}] \leftarrow R^2 \leftarrow R^5 \leftarrow R^3 \leftarrow 0,\]
	\item (Choice of polynomial with minimum signature) $f_5 = x^2y-\frac{1}{2}z^2$ as the remainder of a division of $\Spoly(ze_1,x^2e_2)$ and  $S_{F_4}(f_5) = x^2e_2$.
\end{itemize}

\noindent\underline{Get $f_6$}
\begin{itemize}
	\item (Tuple) $F_5 = (f_1,f_2,f_3,f_4,f_5)$,
	\item (\gr\ obstructions)
	\[\Gobs[\mathit{F}_{5}] = \langle x^2 e_3, xe_5,ye_5,ze_5 \rangle/\langle x^2e_2, ze_3,x^2ye_3,xye_4,xze_5,xye_5,z^3e_5 \rangle,\]
	\item (Minimal free resolution)
	\[0 \leftarrow \Gobs[\mathit{F}_{5}] \leftarrow R^4 \leftarrow R^8 \leftarrow R^4 \leftarrow 0,\]
	\item (Choice of polynomial with minimum signature) $f_6 = xy-\frac{1}{2}y^2$ as the remainder of a division of $\Spoly(xe_2,ze_5)$ and $S_{F_5}(f_6) = ze_5$.
\end{itemize}

\noindent\underline{Get $f_7$}
\begin{itemize}
	\item (Tuple) $F_6 = (f_1,f_2,f_3,f_4,f_5,f_6)$,
	\item (\gr\ obstructions)
	\[ \Gobs[\mathit{F}_{6}] = \langle x^2 e_3, xe_5,ye_5,ze_6,ye_6 \rangle/\left\langle\begin{aligned}  &x^2e_2, ze_3,x^2ye_3,xye_4,\\
	&ze_5,y^2e_5,xye_5,xe_6,z^3e_6 \end{aligned} \right\rangle ,\]
	\item (Minimal free resolution)
	\[0 \leftarrow \Gobs[\mathit{F}_{6}] \leftarrow R^5 \leftarrow R^{11} \leftarrow R^7 \leftarrow R^1 \leftarrow 0,\]
	\item (Choice of polynomial with minimum signature) $f_7 = y^2z-4y$ as the remainder of a division of $\Spoly(e_2,ze_6)$ and $S_{F_6}(f_7) = ze_6$.
\end{itemize}

\noindent\underline{Get $f_8$}
\begin{itemize}
	\item (Tuple) $F_7 = (f_1,f_2,\ldots,f_7)$,
	\item (\gr\ obstructions)
	\[ \Gobs[\mathit{F}_{7}] = \langle x^2 e_3, xe_5,ye_5,ye_6 \rangle/\left\langle\begin{aligned}  &x^2e_2, ze_3,x^2ye_3,xye_4,y^2e_5,\\
	&xye_5,ze_5,ze_6,xe_6,xe_7,z^2e_7 \end{aligned} \right\rangle ,\]
	\item (Minimal free resolution)
	\[0 \leftarrow \Gobs[\mathit{F}_{7}] \leftarrow R^4 \leftarrow R^9 \leftarrow R^6 \leftarrow R^1 \leftarrow 0,\]
	\item (Choice of polynomial with minimum signature) $f_8 = y^3 -2z^2$ as the remainder of a division of $\Spoly(e_3,ye_6)$ and $S_{F_7}(f_8) = ye_6$.
\end{itemize}

\noindent\underline{Get $f_9$}
\begin{itemize}
	\item (Tuple) $F_8 = (f_1,f_2,\ldots,f_8)$,
	\item (\gr\ obstructions)
	\[ \Gobs[\mathit{F}_{8}] = \langle xe_5,xe_8 \rangle/\left\langle\begin{aligned}  &x^2e_2, ze_3,x^2e_3,xye_4,ze_5,ye_5,\\
	&ze_6,ye_6,xe_6,xe_7,z^2e_7,ze_8,xye_8 \end{aligned} \right\rangle ,\]
	\item (Minimal free resolution)
	\[0 \leftarrow \Gobs[\mathit{F}_{8}] \leftarrow R^2 \leftarrow R^4 \leftarrow R^2 \leftarrow 0,\]
	\item (Choice of polynomial with minimum signature) $f_9 =xz^2 - \frac{1}{2} yz^2$ as the remainder of a division of $\Spoly(y^2e_6,xe_8)$ and $S_{F_8}(f_9) = xe_8$.
\end{itemize}

\noindent\underline{Get $f_{10}$}
\begin{itemize}
	\item (Tuple) $F_9 = (f_1,f_2,\ldots,f_9)$,
	\item (\gr\ obstructions)
	\[ \Gobs[\mathit{F}_{9}] = \langle xe_5 \rangle/\left\langle\begin{aligned}  &x^2e_2, ze_3,x^2e_3,xye_4,ze_5,ye_5,ze_6,\\
	&ye_6,xe_6,xe_7,z^2e_7,ze_8,xe_8,ye_9,ze_9 \end{aligned} \right\rangle ,\]
	\item (Minimal free resolution)
	\[0 \leftarrow \Gobs[\mathit{F}_{9}] \leftarrow R^1 \leftarrow R^2 \leftarrow R^1 \leftarrow 0,\]
	\item (Choice of polynomial with minimum signature) $f_{10} =yz^2 -4z$ as the remainder of a division of $\Spoly(e_1,xe_5)$ and $S_{F_9}(f_{10}) = xe_5$.
\end{itemize}

\noindent\underline{Get $f_{11}$}
\begin{itemize}
	\item (Tuple) $F_{10} = (f_1,f_2,\ldots,f_{10})$,
	\item (\gr\ obstructions)
	\[ \Gobs[\mathit{F}_{10}] = \langle xe_{10} \rangle/\left\langle\begin{aligned}  &x^2e_2, ze_3,x^2e_3,xye_4,xe_5,ze_5,ye_5,ze_6,\\
	&ye_6,xe_6,xe_7,z^2e_7,ze_8,xe_8,ye_9,ze_9,ze_{10},ye_{10} \end{aligned} \right\rangle ,\]
	\item (Minimal free resolution)
	\[0 \leftarrow \Gobs[\mathit{F}_{10}] \leftarrow R^1 \leftarrow R^2 \leftarrow R^1 \leftarrow 0,\]
	\item (Choice of polynomial with minimum signature) $f_{11} =xz-\frac{1}{2}yz$ as the remainder of a division of $\Spoly(ye_9,xe_{10})$ and $S_{F_{10}}(f_{11}) = xe_{10}$.
\end{itemize}
\if{
\noindent\underline{Return}\\

The following set of 11 polynomials is a \gr\ basis of $I = \langle f_1,f_2,f_3 \rangle$ since $\Gobs[\mathit{F}_{11}] = 0$:
\[ \begin{aligned}
 f_1 &= x^3 y - z,\\
 f_2 &= xyz - 2y,\\
 f_3 &= xy^2 - z^2,\\
 f_4 & = z^3 - 2y^2,\\
 f_5 & = x^2y - \frac{1}{2}z^2,\\
 f_6 & = xy - \frac{1}{2}y^2,\\
 f_7 & = y^2z - 4y,\\
 f_8 & = y^3 - 2z^2,\\
 f_9 & = xz^2 - \frac{1}{2}yz^2,\\
 f_{10} & = yz^2 - 4z,\\
 f_{11} & = xz - \frac{1}{2}yz.
 \end{aligned}\]
}\fi

\end{document}